\documentclass[12pt,leqno]{amsart}

\usepackage{graphicx,amsmath,amsthm,amsfonts,amssymb,latexsym}

\usepackage[all]{xy}
\usepackage{eucal}
\usepackage{euler}
\usepackage{times}
\usepackage{wasysym,stmaryrd}
\swapnumbers

\newcounter{ictr}
\newenvironment{ilist}{\begin{list}
                         {\textup{(\alph{ictr})}}
                         {\usecounter{ictr}
                          \setlength{\leftmargin}{0.6truein}
                          \setlength{\itemsep}{0.0truein}
                          \setlength{\labelwidth}{0.3truein}}}
                      {\end{list}}

\newcounter{nctr}

\newtheorem*{thmm}{Theorem}

\newtheorem{thm}{Theorem}

\newtheorem{lem}[thm]{Lemma}

\newtheorem{prop}[thm]{Proposition}

\theoremstyle{definition}

\theoremstyle{remark}

\newtheorem*{rem}{Remark}

\numberwithin{equation}{section}

\newcommand{\F}{\mathbb{F}}
\newcommand{\Z}{\mathbb{Z}}
\DeclareMathOperator{\rk}{rk}
\newcommand{\Hawaii}{Hawai\kern.05em`\kern.05em\relax i}
\newcommand{\Manoa}{M\=anoa}

\title{Coarse non-amenability and covers with small eigenvalues}


\author{Goulnara Arzhantseva}
\address{Universit\'{e} de Gen\`{e}ve,
Section de Math\'{e}matiques, 2-4 rue du Li\`{e}vre, Case postale
64, 1211 Gen\`{e}ve 4, Switzerland}
\email{Goulnara.Arjantseva@unige.ch}

\author{Erik Guentner}
\address{University of \Hawaii~at \Manoa,
Department of Mathematics,
2565 McCarthy Mall, Honolulu, HI 96822}
\email{erik@math.hawaii.edu}

\thanks{The research of the first author was partially supported by the Swiss NSF, Sinergia Grant $CRSI22\_130435$.
The second author was partially supported by NSF Grant DMS-0349367.}

\begin{document}

\begin{abstract}
Given a closed Riemannian manifold $M$ and a (virtual) epimorphism $\pi_1(M)\twoheadrightarrow \mathbb{F}_2$ of the
fundamental group onto a free group of rank $2$,
we construct a tower of  finite sheeted regular covers $\left\{M_n\right\}_{n=0}^{\infty}$
of $M$ such that $\lambda_1(M_n)\to 0$ as $n\to\infty$. This is the first example of such a tower which
is not obtainable up to uniform quasi-isometry (or even up to uniform coarse equivalence)
by the previously known methods where $\pi_1(M)$ is supposed to surject onto an amenable group.
\end{abstract}

\maketitle

\section{Introduction}

Let $M$ be a closed (that is, compact and without boundary) Riemannian manifold
with fundamental group $\pi_1(M)$. A residually finite group $G$, a
surjective homomorphism $\pi_1(M)\twoheadrightarrow G$ and a nested sequence of
finite index normal subgroups of $G$ with trivial intersection gives rise
to a tower of  finite sheeted regular covers of $M$; conversely, every tower of
finite sheeted regular
covers arises in this manner. In summary, writing $G_0=G$ and $M_0=M$,
we have:
\begin{equation*}
\label{setup}\tag{$*$}
\begin{minipage}[l]{0.6\linewidth}
  $G_0 \trianglerighteqslant G_1 \trianglerighteqslant G_2 \trianglerighteqslant \cdots$,\;\;\hbox{ with }
  $\displaystyle{\bigcap_{n=0}^{\infty}} G_n= \{\, 1 \,\}$,\;\;\\[2pt]
  $\pi_1(M_n)\twoheadrightarrow G_n, \hbox{ and finite groups } \Gamma_n := G/G_n$,
\end{minipage}
\begin{minipage}[r]{0.4\linewidth}
  \xymatrix{%
     \ar@{..>}[d] &&  \ar@{..>}[d] \\
   \Gamma_2 \ar[d] && M_2 \ar[d] \\
   \Gamma_1 \ar[d] && M_1 \ar[d] \\
   \Gamma_0 && M_0.  }
\end{minipage}
\end{equation*}

In the context of spectral geometry of towers of covers one studies the asymptotic behavior
of the first non-zero eigenvalues $\lambda_1(M_n)$ of the Laplacian,
that is, of the Laplace-Beltrami operator of the individual
Riemannian manifolds $M_n$.  In particular, the following questions are classical:
\begin{ilist}
  \item Does there exist a tower with $\lambda_1(M_n)\geqslant c>0$ uniformly over $n$?
  \item Does there exist a tower with $\lambda_1(M_n)\to 0$ as $n\to\infty$?
\end{ilist}

In this note we are concerned with (b). The earliest positive result on this
question is due to Randol, who studied the case of cyclic covers using
the trace formula \cite{MR0400316}. Subsequent results of Brooks
\cite{MR840402,MR799796} and Burger \cite{MR832070} were obtained by
relating the
eigenvalues $\lambda_1(M_n)$ to combinatorial properties of the Cayley graphs of
finite groups of deck transformations $\Gamma_n$.  Similar results are due to Sunada~\cite{MR782558}.

In all cases, the method to build a tower of covers satisfying (b)
rests on choosing an {\it amenable\/} group $G$ for the construction~(\ref{setup}). Our main result is that it is possible to obtain such a
tower when $G$ is the free group on two generators.
In the statement, $H^{(2)}$ denotes the subgroup of the discrete group
$H$ generated by the squares of its elements.

\begin{thmm}
  Let $M$ be a closed Riemannian manifold, whose fundamental group
  admits a virtual\footnote{A virtual homomorphism is a homomorphism of a finite-index subgroup.} surjective homomorphism onto
  the free group
   of rank 2.  Taking the nested sequence of
  subgroups in \textup{(}\ref{setup}\textup{)}
  to be the sequence of iterated squares in the free group
\begin{equation*}
\label{squares}\tag{$**$}
 G_0=\F_2,\quad G_1 = \F_2^{(2)}, \quad G_2 =(\mathbb{F}_2^{(2)})_{\strut}^{(2)},
       \quad G_3 =((\mathbb{F}_2^{(2)})_{\strut}^{(2)})_{\strut}^{(2)},\quad \cdots
\end{equation*}
  we obtain a tower of covers of
  $M$ for which $\lambda_1(M_n)\to 0$ as $n\to\infty$.   This tower is not
  obtainable up to uniform quasi-isometry \textup{(}or even uniform
  coarse equivalence\textup{)} by the construction
  \textup{(}\ref{setup}\textup{)} with an amenable $G$.
\end{thmm}

\noindent
Observe here that each $G_n$
is normal, even characteristic, in $\F_2$.

The hypothesis of the theorem means that the fundamental group is
{\it large\/} (the terminology is due to Gromov~\cite{gromov}). It
applies to many hyperbolic manifolds~\cite{l}, in particular, to a
closed orientable surface of genus at least two -- the fundamental
group of such a manifold surjects onto $\F_2$.

We conclude the introduction by remarking that in more modern
terminology the classical problems above concerning the construction
(\ref{setup}) can be rephrased in terms of {\it Property\/ $\tau$}:
(a) asks for $G$ to have Property $\tau$ with respect to the family of
subgroups $(G_n)_{n\geqslant 0}$, whereas (b) asks, after perhaps passing to
a subsequence, for $G$ to {\it not\/} have Property $\tau$ with
respect to the $(G_n)_{n\geqslant 0}$.  This is explained in the work of
Burger and Brooks cited above.  Thus, the first assertion in the
theorem is essentially equivalent to the assertion that $\F_2$ does
not have Property $\tau$ with respect to the subgroups appearing in
(\ref{squares}).  For the definition and relevant facts about Property
$\tau$ see \cite{lubotzky}.

\section{Eigenvalues}

A {\it graph\/} is a collection of {\it vertices\/} and {\it edges\/}.
With a small number of exceptions, we permit neither multiple edges
nor loops, so that an edge is uniquely determined by its incident
vertices.  Our graphs are unoriented.
The {\it Cheeger constant\/} of a finite graph $\Gamma$ is
\begin{equation}
\label{cheeger}
  h(\Gamma) = \inf \frac{\# E(A,B)}{\min\{\, \# A, \# B \,\}},
\end{equation}
where the infimum is taken over all decompositions of the vertex set
of $\Gamma$ as a disjoint union $A\sqcup B$ and where, for such a
decomposition, $E(A,B)$ denotes the set of edges with one incident
vertex in $A$ and the other in $B$.

We shall make use of the following result of Brooks which, in the
notation of (\ref{setup}), relates the eigenvalues of the $M_n$ to the
Cheeger constants of the Cayley graphs of the $\Gamma_n$
computed  with respect to the canonical images of generators of $G$ and denoted, by an abuse of notation, again by $\Gamma_n$.
  We shall require only the forward implication, which is the content of
\cite[Lemma~1]{MR840402}.

\begin{thmm}[Brooks]
  In the notation of
  \textup{(}\ref{setup}\textup{)} we have
  $h(\Gamma_n)\to 0$  precisely when $\lambda_1(M_n)\to 0$.  \qed
\end{thmm}

Thus, the first statement in theorem of the introduction is reduced to
the following:

\begin{prop}\label{prop:sq}
  Let $G=\F_2$ be the free group of rank 2.  Consider the tower of
  iterated squares \textup{(}\ref{squares}\textup{)} and the
  corresponding quotients:
\begin{equation*}
    \Gamma_0 = \{\, 1 \,\} \leftarrow \Gamma_1 = \F_2/\F_2^{(2)}
      \leftarrow \Gamma_2 = \F_2/(\F_2^{(2)})_{\strut}^{(2)} \leftarrow \cdots
  \end{equation*}
  Abusing notation, view each $\Gamma_n$ as a Cayley graph
  with respect to the images of the standard free generators of $\F_2$. Then we
  have $h(\Gamma_n)\to 0$ as $n\to\infty$.
\end{prop}

In preparation for the proof we recall the construction of the {\it
$\Z/2$-homology cover\/} of a finite graph $\Sigma$.  Fix a maximal
tree $T$ in $\Sigma$ and let $e_1,\dots,e_r$ be the edges of
$\Sigma$ {\it not\/} in $T$.  The vertex and edge sets of the
$\Z/2$-homology cover $\widetilde\Sigma$ are
\begin{equation*}
  \widetilde V = V \times \oplus_1^r \Z/2, \quad
  \widetilde E = E \times \oplus_1^r \Z/2,
\end{equation*}
where $E$ and $V$ denote the vertex and edge sets of $\Sigma$. Let
$e\in E$ and let $v$, $w\in V$ be the vertices incident with $e$.
Consider the edge $(e,\alpha)\in\widetilde E$. Incidence is defined
in two cases:
\begin{equation*}
  \text{$(e,\alpha)$ contains} \begin{cases}
          \text{$(v,\alpha)$ and $(w,\alpha)$},
                &\text{when $e$ belongs to the maximal tree $T$} \\
          \text{$(v,\alpha)$ and $(w,\alpha+\overline{e}_j)$},
                &\text{when $e=e_j$, for some $1\leqslant j \leqslant r$.}
                          \end{cases}
\end{equation*}
Here $\overline{e}_j=(\dots,1,\dots)$ is the standard basis vector
with a single $1$ in the $j$-the position and $0$'s elsewhere.
Strictly speaking, when defining incidence it is necessary to {\it
  direct\/} the edges $e_j$. It is quickly verified however that,
while the edges are parameterized in a different manner, the
underlying {\it undirected\/} graph is independent of the choice. We
shall not dwell on this aspect.

\begin{rem}
  The construction given here of the $\Z/2$-homology cover is a
  special case of the classical construction of a finite sheeted regular cover of
  $\Sigma$ corresponding to
  a given normal subgroup of finite index in  $\pi_1(\Sigma)$, see, for example,~\cite[Ch. 2]{stillwell}.
  Indeed, with $e_1,\dots,e_r$ as above, and after directing each
  $e_j$, we identify
\begin{equation*}
 \pi_1(\Sigma) \cong \F_r = \langle\; e_1,\dots,e_r \;\rangle.
\end{equation*}
Then the cover corresponding to the kernel of the epimorphism
\begin{equation*}
      \pi_1(\Sigma)\cong \F_r \twoheadrightarrow {\F_r}/{\F_r^{(2)}}\cong
\oplus_1^r \Z/2\quad \hbox{ defined by } \quad
e_j\mapsto\overline{e}_j,
 \end{equation*}
is the $\Z/2$-homology cover.
\end{rem}

\begin{lem}\label{lem:ch}
  Let $\Sigma$ be a finite graph, with vertex set $V$; let
  $\widetilde\Sigma$ be its
  $\Z/2$-homology cover.  We have
  \begin{equation*}
    h(\widetilde\Sigma)\leqslant \frac{2}{\# V}.
 \end{equation*}
\end{lem}

\begin{proof}
  We employ the notation introduced above for $\widetilde\Sigma$.
  We shall exhibit a
  decomposition of the vertex set $\widetilde V = A \sqcup B$ for which the
  quotient in  (\ref{cheeger}) is bounded by $2/\# V$.  Let
  \begin{equation*}
    A = \{\, (v,\alpha)\in \widetilde V: \alpha=(*,\dots,*,0) \;\}, \quad
    B = \{\, (w,\beta)\in \widetilde V : \beta=(*,\dots,*,1) \;\},
  \end{equation*}
each of which contains exactly  $2^{r-1}\, \# V$ vertices. The edges
in $\widetilde E$ with one vertex in $A$ and the other in $B$ are
exactly those of the form $(e_r,\gamma)$, for arbitrary
$\gamma\in\oplus_1^r \Z/2$; thus $E(A,B)$ contains exactly $2^r$
edges.
\end{proof}

\begin{proof}[Proof of Proposition~\ref{prop:sq}]
  The Cayley graph $\Gamma_n$ is the $n$-th iterated $\Z/2$-homology
  cover of the ``figure $8$''.  Since the number of vertices in
  $\Gamma_n$ tends to infinity, the result follows from the previous
  lemma.
\end{proof}

\begin{rem}
  A more detailed analysis gives information on the {\it rate\/} of
  the convergence $h(\Gamma_n)\to 0$.  Indeed, let $V_n$ be the set
  of vertices and $E_n$ the set of edges of (the Cayley graph of)
  $\Gamma_n$.  We have
  \begin{equation*}
    \frac{\# V_{n+1}}{\# V_{n}} =    \frac{\# E_{n+1}}{\# E_{n}}
         = 2^{\rk \pi_1(\Gamma_{n})}.
  \end{equation*}
  Now, the rank
  of the fundamental group $\pi_1(\Gamma_{n})$ is
  the number of edges
  {\it not\/} belonging to a fixed maximal tree in $\Gamma_{n}$.
  Since $\# E_n = 2 \cdot\# V_n$, the rank of $\pi_1(\Gamma_{n})$ is
  $\# V_n+1$.
  Thus, we get the recursive formula
\begin{equation*}
  \# V_{n+1} = \# V_n \cdot 2^{\# V_n +1}.
\end{equation*}
In particular, $\# V_n$ grows faster than an iterated exponential and,
according to the previous lemma, the Cheeger constant
$h(\Gamma_{n+1})$ decays as the reciprocal of $\# V_n$.
\end{rem}

\section{Non uniform coarse equivalence}

We shall now show that the tower constructed in the previous section
cannot be duplicated beginning with an amenable group in
(\ref{setup}), thus completing the proof of the theorem in the
introduction.

 Two families $(X_n)_{n\geqslant 0}$ and $(Y_n)_{n\geqslant 0}$ of metric spaces are {\it
  uniformly   quasi-isometric\/} if there exist functions $f_n:X_n\to Y_n$ and
constants $C\geqslant 1$ and $D\geqslant 0$ such that for all
$x,y\in X_n$ and $z\in Y_n$, we have
\begin{itemize}
  \item $C^{-1} d(x,y) - D \leqslant d(f_n(x),f_n(y)) \leqslant C d(x,y) + D$,
    \item $d(z, f_n(X_n))\leqslant D$.
\end{itemize}
The families $(X_n)_{n\geqslant 0}$ and $(Y_n)_{n\geqslant 0}$ are {\it uniformly coarsely equivalent\/}
if there exist functions $f_n:X_n\to Y_n$ with the following two
properties:
\begin{itemize}
  \item $\forall A\; \exists B$ such that
           $\forall n\;\forall x,y\in X_n$ we have
           $d(x,y)\leqslant A\ \Rightarrow\ d(f_n(x),f_n(y))\leqslant B$,
  \item $\forall A\; \exists B$ such that
           $\forall n\;\forall x,y\in X_n$ we have
           $d(x,y)\geqslant B\ \Rightarrow\ d(f_n(x),f_n(y))\geqslant A$.
\end{itemize}
If two families are uniformly quasi-isometric then they are uniformly
coarsely equivalent.  Observe that these notions apply to individual
spaces, which we regard as trivial families containing a single space.
We say, for example, two spaces are coarsely equivalent.

\begin{prop}[The Uniform \v{S}varc-Milnor Lem\-ma]\label{prop:uqi}
  Continue with the notation of
  \textup{(}\ref{setup}\textup{)}.  Equip each $\Gamma_n$ with the
  word metric associated to a fixed finite
  generating set for $G$; equip each $M_n$ with the path metric
  associated to its Riemannian structure.  The families $(\Gamma_n)_{n\geqslant 0}$ and
  $(M_n)_{n\geqslant 0}$ are uniformly quasi-isometric.
\end{prop}

\begin{rem}
  In the situation of (\ref{setup}) the group $G$ is indeed
  finitely generated.  Further, the statement in the proposition is
  independent of the choice of generators for $G$.
\end{rem}

\begin{proof}[Proof of Proposition~\ref{prop:uqi}]
  The result follows from the \v{S}varc-Milnor Lem\-ma \cite[Prop.~I.8.19]{bh},
  observing that the inherent quasi-isometry constants (see the proof of the Lemma) depend
  only on the diameter of a fundamental domain for the action. In
  detail, $\Gamma_n$ is the group of deck transformations of the cover
  $M_n$ of $M$, whereas $G$ is the group of deck transformations
  of the cover corresponding to the kernel of the surjective
  homomorphism $\pi_1(M) \twoheadrightarrow G$.  Further, the image in
  $M_n$ of a bounded fundamental
  domain for the action of $G$ is a
  fundamental domain for the action of $\Gamma_n$, of no greater diameter.
\end{proof}

Thus, the second statement in the theorem of the introduction is
reduced to the following:

\begin{prop}\label{prop:nqe}
  Consider the tower of iterated squares
  \textup{(}\ref{squares}\textup{)} of the free
  group $\F_2$ and the corresponding quotients
  \begin{equation*}
    \Gamma_0 = \{\, 1 \,\} \leftarrow \Gamma_1 = \F_2/\F_2^{(2)}
      \leftarrow \Gamma_2 = \F_2/(\F_2^{(2)})_{\strut}^{(2)} \leftarrow \cdots
  \end{equation*}
Then the family $(\Gamma_n)_{n\geqslant 0}$ is not uniformly coarsely equivalent to any family of
quotients of an amenable group.
\end{prop}

Let $G$ be a finitely generated discrete group, and let $\ell$ be the
word length associated to a fixed finite and symmetric set of
generators. Of the many equivalent definitions of amenabilitiy we
shall work with {\it Reiter's condition\/} -- $G$ is {\it amenable\/}
if for every $\varepsilon>0$ and for every $R>0$ there exists a
finitely supported
$\xi\in\ell^1(G)$ such that $\xi\geqslant 0$, $\|\xi\|=1$ and
\begin{equation}
\label{amen}
   \ell(g)\leqslant R \Rightarrow \| g\cdot \xi - \xi \| <\varepsilon,
\end{equation}
where the action of $G$ on $\ell^1(G)$ is defined by $g\cdot \xi(h)=\xi(g^{-1}h).$

Our main tool to prove Proposition~\ref{prop:nqe} is the use of Property A, a weak form of amenability, introduced
by Yu in the context of the Baum-Connes conjecture in topology~\cite{Yu}.

Let $X$ be a discrete metric space of {\it bounded geometry\/} -- that
is, the number of points in a ball of fixed radius is bounded, the
bound depending only on the radius of the ball and not on its center.
Of the many equivalent definitions of Property $A$ we choose
the one most closely related to Reiter's condition --  $X$ has {\it  Property
$A$\/} if for every $\varepsilon>0$ and $R>0$ there exists an $S>0$ and
for each $x\in X$ a function $\xi_x\in \ell^1(X)$ such that $\xi_x\geqslant
0$, $\|\xi_x\|=1$ and \begin{align*}
  d(x,y)\leqslant R &\Rightarrow \| \xi_x -\xi_y \| < \varepsilon, \\
  \xi_x(y) \neq 0 &\Rightarrow d(x,y)\leqslant S.
\end{align*}
The analogy with amenability being clear, we say that a metric space
having Property $A$ is {\it coarsely amenable\/} whereas one not
having Property $A$ is {\it coarsely non-amenable}.

Finally, a metric space $X$ is the {\it coarse union\/} of its
subspaces $X_n$ if $X=\sqcup X_n$ (disjoint union), and if
$d(X_n,X_m)\to \infty$ as $n+m\to\infty$. If the $X_n$ are metric
spaces each having finite diameter, then there exists a metric space
$X$ which is the coarse union of (isometric copies of) the
$X_n$. Further, any
two such unions are coarsely equivalent.  Moreover, if $Y$ is the coarse
union of the $Y_n$ then $X$ and $Y$ are coarsely equivalent
when the $X_n$ and $Y_n$ are uniformly coarsely equivalent.

We require the following slight
generalization of \cite[Prop. 11.39]{roe}.  We include a proof which is
both different from other proofs in the literature and
convenient for our result.

\begin{prop}\label{prop:cu}
  Let $G$ be a finitely generated amenable group.
  Every quotient of $G$ is amenable; the coarse union of any family of finite
  quotients of $G$ is coarsely amenable.
\end{prop}

\begin{proof}
  Let $H$ be a quotient of $G$ and identify $H$ with a set of cosets
  $\{\, gN \,\}$.  Fix a finite and symmetric set of
  generators for $G$ and equip $G$ with the associated word length;
  equip $H$ with the word length associated to the induced generators.
  With these conventions
  \begin{equation*}
    \ell_H(x)\leqslant R\ \Longleftrightarrow\
       \text{$\exists \; g\in x$ such that $\ell_G(g)\leqslant R$}
  \end{equation*}
  and, in particular, the map $G\twoheadrightarrow H$ is contractive. Given
  $\varepsilon>0$ and $R>0$, obtain $\xi\in\ell^1(G)$ as in
  (\ref{amen}).  Define
  \begin{equation}
\label{eta}
    \eta(x) = \sum_{g\in x} \xi(g),
  \end{equation}
  so that $\eta\geqslant 0$ and $\| \eta\|=1$.  Further, when $z\in H$ has
  length at most $R$ we obtain $g\in G$ of length at most $R$ such
  that $z=gN$.  We then calculate
  \begin{equation*}
  \| z\cdot \eta -\eta \| = \sum_{x\in H} | \eta(g^{-1}x)-\eta(x)|
       \leqslant \sum_{x\in H} \sum_{h\in x} | \xi(g^{-1}h)-\xi(h)|
       = \| g\cdot \xi-\xi\| <\varepsilon.
\end{equation*}
We conclude that $H$ is amenable.

When dealing with a coarse union the essential observation is that, in
the previous argument, if $\xi$ is supported on the elements of length
at most $S$ then the same is true of $\eta$. Thus, let $H_n$ be a
family of quotients of $G$, each equipped with a length function as
above, and let $X$ be a coarse union of the $H_n$. Given
$\varepsilon>0$ and $R>0$ proceed as above -- obtain a Reiter function
$\xi$ for $G$ and define $\eta_n$ as in (\ref{eta}). For $x\in X$
define
\begin{equation*}
  \xi_x = \begin{cases}
        \chi_N, &\text{$x\in H_n$, $n\leqslant N$}\\
        x\cdot \eta_n, & \text{$x\in H_n$, $n>N$},
  \end{cases}
\end{equation*}
where $N$ is chosen large enough so that for $n>N$ the distance
between $H_n$ and any other $H_m$ is at least $R$; $\chi_N$ is the
characteristic function of $H_1\cup\dots\cup H_N$.  Finally, choose
$S$ larger than the diameter of $H_1\cup\dots\cup H_N$ and large
enough so that $\xi$ is supported on elements of length at most $S$ in
$G$.

The required properties are easily verified.
\end{proof}

\begin{proof}[Proof of Proposition~\ref{prop:nqe}]
  The  iterated squares are proper characteristic subgroups of the free group, hence,
by Levi's theorem~\cite[Ch.I, Prop.~3.3]{ls},    they have trivial intersection,
  $\cap \F_2^{(2)\ldots(2)}= \{\, 1 \,\}.$
     Thus, the coarse union of the metric
  spaces $\Gamma_n$ is an example of a coarsely non-amenable {\it box
    space\/}. See \cite[Def. 11.24 and Prop. 11.39]{roe}.  (This statement is the converse of the
  previous proposition, and can also be proved by modifying the above
  argument.) This finishes the proof as coarse amenability is invariant under coarse equivalence,
  see, for example,~\cite[Prop. 4.2]{Tu}.
\end{proof}

We conclude with two remarks.  First, we have used a very crude
invariant from coarse geometry to distinguish towers constructed
from the sequence of iterated squares (\ref{squares}) from those
constructed beginning with an amenable group in (\ref{setup}) -- the
former are coarsely non-amenable while the latter are coarsely
amenable.  More refined invariants would be needed to establish the
existence of coarsely inequivalent towers constructed as in
(\ref{setup}) from a given non-amenable group.

Second, our construction involving the iterated squares
(\ref{squares}) is particular to the free group.  It would be
interesting to remove the hypothesis of `largeness' from our theorem.


\bibliographystyle{plain}
\bibliography{evalsA}

\end{document}